\documentclass[11pt]{amsart}

\usepackage{amsmath,amssymb,amsthm,mathrsfs}
\usepackage{esint}
\usepackage{mathtools}
\mathtoolsset{showonlyrefs}
\usepackage[margin=1.15in]{geometry}
\usepackage[colorlinks=true,citecolor=blue,linkcolor=blue,urlcolor=blue]{hyperref}

\newtheorem{thm}{Theorem}[section]
\newtheorem{lem}[thm]{Lemma}
\newtheorem{coro}[thm]{Corollary}
\newtheorem{prop}[thm]{Proposition}

\numberwithin{equation}{section}

\theoremstyle{definition}
\newtheorem{defn}[thm]{Definition}
\theoremstyle{remark}

\begin{document} 
	
	\title[Positive Mass Theorem with Arbitrary Ends]{Positive Mass Theorem with Arbitrary Ends and Noncompact Boundary}
	\date{\today}
	
	\author{Caiyan Li}
	\address[C.L]{School of Mathematical Sciences, Xiamen University, 361005, Xiamen, P.R. China}
	\email{caiyanli@xmu.edu.cn}
	
	\begin{abstract}
		We prove a positive mass theorem for complete Riemannian manifolds with
		noncompact boundary, a distinguished asymptotically flat half-space end, and
		finitely many additional complete ends with no prescribed asymptotics. If
		$3\leq n\leq7$, $R_g\geq0$, and $H_{\partial M}\geq0$, then $ \mathfrak m(M,g,\mathcal E)\geq0$. 
		Moreover, equality holds if and only if $(M,g)$ is isometric to the Euclidean half-space. The proof combines a density deformation near the distinguished
		end with doubling across the noncompact boundary, local smoothing, and a
		conformal correction. We also obtain the sharp Riemannian Penrose inequality
		when a compact outermost minimal hypersurface separates $\mathcal E$ from all
		the remaining ends; equality holds precisely when the exterior region is a
		Schwarzschild half-space exterior.
	\end{abstract}
	
	\maketitle
	
	\section{Introduction}
	
	The Riemannian positive mass theorem states that a complete asymptotically
	flat manifold with nonnegative scalar curvature has nonnegative ADM mass, and
	that vanishing mass characterizes Euclidean space. Schoen and Yau established
	the theorem in dimensions $3\leq n\leq7$ by minimal-hypersurface methods
	\cite{SY79,SY81}, while Witten proved it in all dimensions under a spin
	assumption \cite{Wit81}; see also \cite{Bart86}.
	
	For asymptotically flat manifolds with noncompact boundary, Almaraz, Barbosa,
	and Lopes de Lima \cite{ABL16} developed the corresponding notion of mass and
	proved positivity and rigidity when the scalar curvature and the boundary mean
	curvature are nonnegative, either in dimensions $3\leq n\leq7$ or under a spin
	assumption. They also established a density theorem adapted to half-space
	asymptotics and a variational formula for the mass; both will be used below.
	Chai \cite{Chai18} gave an alternative proof in dimensions
	$3\leq n\leq7$ using free-boundary minimal hypersurfaces. In dimension three,
	Batista and de Lima \cite{BdL25} obtained a proof based on harmonic level sets.
	
	These results lead to a natural localization question: does the mass of a distinguished
	asymptotically flat end remain nonnegative when no asymptotic conditions are
	imposed on the other ends? In the spin setting, this question was studied by
	Bartnik and Chru\'sciel \cite{BC03,BC05}, and in quantitative form by Cecchini
	and Zeidler \cite{CZ24}. In dimensions at most seven, Lesourd, Unger, and Yau
	\cite{LUY24} first obtained such a result under an asymptotically Schwarzschild
	assumption on the distinguished end. Zhu \cite{Zhu23} subsequently established
	positivity and rigidity under general asymptotically flat decay. In the presence
	of a noncompact boundary, Liu \cite{Liu24} proved a spacetime positive mass
	theorem with arbitrary ends under a spin assumption.
	
	The aim of this paper is to establish the corresponding non-spin Riemannian
	result for manifolds with noncompact boundary in dimensions
	$3\leq n\leq7$. We require half-space asymptotics only along one distinguished
	end and impose no asymptotic structure on the finitely many additional ends,
	which are assumed merely to be complete. This extends \cite{ABL16} beyond the
	one-ended setting and gives a boundary analogue of \cite{LUY24,Zhu23}. The
	essential point is that no geometric control is required on the remaining
	ends or on the noncompact boundary along them.

	We now describe the geometric setting and state the main results precisely.
	Let $(M,g)$ be a smooth, complete, noncompact Riemannian manifold with boundary
	$\partial M$. Assume that $M$ has a distinguished end $\mathcal E$ and
	finitely many additional complete ends.  
	
	\begin{defn}\label{Def A.F.}
		We say that $(M,g)$ is \emph{asymptotically flat along $\mathcal E$} if
		$\mathcal E$ is diffeomorphic to
		$\mathbb R^n_+\setminus\overline{\mathbb B}_1^+$ and, in the corresponding
		coordinates,
		\[
		g_{ij}=\delta_{ij}+h_{ij},
		\]
		where
		\[
		|h|+r|\partial h|+r^2|\partial^2h|\leq Cr^{2-n},
		\qquad r=|x|.
		\]
		We further assume that
		\[
		R_g=O(r^{-q})
		\]
		on $\mathcal E$ for some $q>n$, and that
		\[
		H_{\partial M}\in L^1(\partial M\cap\mathcal E).
		\]
		No asymptotic assumptions are imposed on the
		remaining ends, which are required only to be complete. We refer to this as the
		\emph{asymptotically flat distinguished-end setting with arbitrary remaining
			ends}.
	\end{defn}
	
	\begin{defn}\label{Def mass}
		The mass of $(M,g,\mathcal E)$ is defined from the asymptotic geometry of the
		distinguished end by
		\begin{equation}
		\begin{aligned}
			\mathfrak m(M,g,\mathcal E)
			&=\frac{1}{2(n-1)\omega_{n-1}}\lim_{r\to\infty}
			\Bigg\{
			\int_{\mathbb S_{r,+}^{n-1}}
			(g_{ij,j}-g_{jj,i})\nu^i\,d\mathbb S_{r,+}^{n-1}+
			\int_{\mathbb S_r^{n-2}}
			g_{\alpha n}\eta^\alpha\,d\mathbb S_r^{n-2}
			\Bigg\}.
		\end{aligned}
		\end{equation}
		Here $\omega_{n-1}=|\mathbb S^{n-1}|$,
		$\mathbb S_{r,+}^{n-1}$ is the coordinate hemisphere of radius $r$, $\nu$ is
		its Euclidean outward unit normal, and $\eta$ is the Euclidean outward unit
		conormal of
		$\mathbb S_r^{n-2}=\partial\mathbb S_{r,+}^{n-1}$ in the truncated boundary
		region $(\partial M)_r=\partial M\cap\{|x|\leq r\}$.
	\end{defn}
	
	Throughout the paper, we use the Einstein summation convention, with Roman
	indices $i,j,\ldots$ ranging from $1$ to $n$ and Greek indices
	$\alpha,\beta,\ldots$ ranging from $1$ to $n-1$. Along $\partial M$, the
	vectors $\{\partial_\alpha\}$ are tangent to the boundary, whereas
	$\partial_n$ points inward.
	
	We can now state the main result.
	
	\begin{thm}\label{Thm: main A}
		Let $3\leq n\leq7$, and let $(M,g,\mathcal E)$ be asymptotically flat along the
		distinguished end $\mathcal E$, with finitely many arbitrary additional
		complete ends. If
		\[
		R_g\geq0 \quad\text{on }M,
		\qquad
		H_{\partial M}\geq0 \quad\text{on }\partial M,
		\]
		then
		\[
		\mathfrak m(M,g,\mathcal E)\geq0.
		\]
		Moreover, equality holds if and only if
		\[
		(M,g)\cong(\mathbb R^n_+,\delta),
		\]
		with $\mathcal E$ corresponding to the standard end of the Euclidean
		half-space.
	\end{thm}
	
	The proof of nonnegativity is localized near the distinguished end. We first apply a density
	deformation modeled on \cite{ABL16} so that the metric is conformally flat and
	scalar-flat near infinity and the boundary is minimal there. We then double
	across the noncompact boundary. Since the doubling hypersurface need not admit
	a uniform Gaussian collar, the corner is smoothed only in a prescribed
	neighborhood of the distinguished end; away from that neighborhood, the double
	is completed using an arbitrary smooth complete metric. A conformal correction
	removes the small negative part of the scalar curvature in the protected region
	without affecting the limiting mass. The quantitative localization result
	\cite[Corollary~1.6]{LLU} then rules out negative mass. In the equality case, a
	compactly supported variation of the metric gives ${\rm Ric}_g=0$ and $A_g=0$,
	after which the rigidity statement follows by doubling and applying
	\cite[Theorem~1.2]{Zhu23}.

	If a compact outermost minimal hypersurface separates $\mathcal E$ from all
	the additional ends, then the exterior region containing $\mathcal E$ is a
	one-ended asymptotically flat half-space. The sharp Penrose inequality and its
	rigidity statement therefore follow directly from Eichmair and Koerber
	\cite{EK23}. We record this consequence for completeness; the new content of
	the paper is Theorem~\ref{Thm: main A}.
	
	\begin{coro}
		Let $3\leq n\leq7$, and let $(M,g,\mathcal E)$ satisfy the distinguished-end
		asymptotic assumptions above. Suppose that a nonempty compact separating
		minimal hypersurface $\Sigma$, whose components are either closed or
		free-boundary, separates $\mathcal E$ from all the remaining ends. Let
		$M_{\mathcal E}(\Sigma)$ denote the closure of the component of
		$M\setminus\Sigma$ containing $\mathcal E$. Assume that $\Sigma$ is outermost
		relative to $\mathcal E$, in the sense that every compact minimal closed or
		free-boundary hypersurface in $M_{\mathcal E}(\Sigma)$ is a component of
		$\Sigma$. If
		$R_g\geq0$ on $M_{\mathcal E}(\Sigma)$ and
		$H_{\partial M}\geq0$ on
		$\partial M\cap M_{\mathcal E}(\Sigma)$, then
		\[
		\mathfrak m(M,g,\mathcal E)
		\geq
		2^{-\frac n{n-1}}
		\left(
		\frac{|\Sigma|_g}{\omega_{n-1}}
		\right)^{\frac{n-2}{n-1}}.
		\]
		Equality holds if and only if
		$\bigl(M_{\mathcal E}(\Sigma),g\bigr)$ is isometric to a Schwarzschild
		half-space exterior.
	\end{coro}
	
	\begin{proof}
		Since $\Sigma$ separates $\mathcal E$ from every other end, the exterior
		region $M_{\mathcal E}(\Sigma)$ contains no end other than $\mathcal E$.
		With the induced metric, it is therefore an asymptotically flat half-space
		with horizon boundary $\Sigma$ in the sense of \cite{EK23}. Moreover,
		\[
		\mathfrak m\bigl(M_{\mathcal E}(\Sigma),g\bigr)
		=\mathfrak m(M,g,\mathcal E).
		\]
		The inequality follows from \cite[Corollary~10]{EK23}, and the equality
		characterization follows from \cite[Theorem~12]{EK23}.
	\end{proof}

	\subsection*{Acknowledgments}

	The author would like to thank Jintian Zhu for helpful discussions. This work
	was partially supported by NSFC (Grant No.~12501274).

	\section{Preparatory constructions}
	\label{sec:preparatory}

We begin with the analytic and geometric constructions used in the proof.
Throughout the paper, subscripts indicate the metric with respect to which
geometric quantities are computed. Set
$a_n=4(n-1)/(n-2)$ and $b_n=2(n-1)/(n-2)$, and let
$L_g=-a_n\Delta_g+R_g$ and
$B_g=b_n\partial_{\eta_g}+H_g$ denote the conformal Laplacian and the
conformal boundary operator, respectively. For every positive function $u$,
the metric $g_{u}=u^{\frac{4}{n-2}}g$ satisfies
\[
R_{g_{u}}=u^{-\frac{n+2}{n-2}}L_gu,
\qquad
H_{g_{u}}=u^{-\frac{n}{n-2}}B_gu.
\]

We also fix a smooth positive function $r$ on $M$ that agrees with the Euclidean radius $|x|$ sufficiently far out along the distinguished end
$\mathcal E$.

\subsection{Conformal deformation at the distinguished asymptotically flat end}
Our first task is to refine the metric near the distinguished asymptotically
flat end $\mathcal E$ without imposing any conditions on the remaining ends.

The required conformal factor will be obtained from the following Sobolev
inequality \cite{SY79,Zhu23} and Sobolev trace inequality
\cite[Theorem 3.3]{Escobar90}.
\begin{lem}\label{Lem: Sobolev}
	Let $U\supset\mathcal E$ be an open neighborhood of the distinguished
	asymptotically flat end $ \mathcal E$ such that
	$\overline{U\setminus\mathcal E}$ is compact in $M$. Then there exist positive constants $c_S$ and $c_T$, depending only on $U$ and $g$, such that 
	\begin{equation}\label{Eq: Sobolev}
		c_S\left(\int_{  U}|\phi|^\frac{2n}{n-2}\,\mathrm d\mu_g\right)^{\frac{n-2}{n}}\leq \int_{  U} |\nabla\phi|^2\,\mathrm d\mu_g, 
	\end{equation}
	and  
	\begin{equation}\label{Eq: Sobolev trace}
		c_T\left(\int_{\partial M\cap U}|\phi|^\frac{2(n-1)}{n-2}\,\mathrm d\sigma_g\right)^{\frac{n-2}{n-1}}\leq \int_{U}|\nabla\phi|^2\,\mathrm d\mu_g, 
	\end{equation}
	for every $\phi\in C_c^\infty( \bar U)$.
\end{lem} 

With these inequalities in hand, we establish the required solvability result.
\begin{prop}\label{Prop: the conformal factor}
	Let $(M,g,\mathcal E)$ be a $C^2$ asymptotically flat manifold with arbitrary ends. Let $U$, $c_S$, and $c_T$ be as in Lemma~\ref{Lem: Sobolev}. Assume that $f\in C^\infty(M)$ and $\tilde f\in C^\infty(\partial M)$ have compact support contained in $U $ and $\partial M \cap U$, respectively, and that their negative parts $f_-$ and $\tilde f_-$ satisfy
	\begin{equation}\label{Eq: f slight negative}
		\left(\int_{U }|f_-|^{\frac{n}{2}}\,\mathrm d\mu_g\right)^{\frac{2}{n}}
		\le \frac{c_S}{4},
		\qquad
		\left(\int_{ \partial M \cap U}|\tilde f_-|^{n-1}\,\mathrm d\sigma_g\right)^{\frac{1}{n-1}}
		\le \frac{c_T}{4}.
	\end{equation}
	Then
	\begin{equation}\label{system u}
		\left\{
		\begin{aligned}
			\Delta_g u-fu&=0, &&\text{in }M,\\
			\frac{\partial u}{\partial\mathbf n}+\tilde f u&=0, &&\text{on }\partial M,\\
			u(x)&\to1, &&\text{as }r(x)\to\infty\text{ along }\mathcal E,
		\end{aligned}
		\right.
	\end{equation}
	admits a positive solution $u$. Moreover, along the asymptotically flat end
	$\mathcal E$, we have $u=1+Ar^{2-n}+\omega$, where
	\[
	A=-\frac{2}{(n-2)|\mathbb S^{n-1}|}\left(\int_{U }fu\,\mathrm d\mu_g+\int_{\partial M \cap U}\tilde f u\,\mathrm d\sigma_g\right),
	\]
	and $|\omega|+r|\partial\omega|+r^2|\partial^2\omega|\le Cr^{1-n}$.
	In addition,
	\[
	\int_{\partial U\cap \mathring{M}}\frac{\partial u}{\partial\mathbf n'}\,\mathrm d\sigma_g=0,
	\qquad
	\int_{\partial U\cap \mathring{M}}u\frac{\partial u}{\partial\mathbf n'}\,\mathrm d\sigma_g\le0.
	\]
	The solution selected by the exhaustion construction also satisfies
	$\inf_Mu>0$.
	Here $\mathbf n$ denotes the outward unit normal to $M$ along $\partial M$, and $\mathbf n'$ denotes the outward unit normal to the domain $U$ along $\partial U\cap M$. 
\end{prop}   
\begin{proof}
	First, choose an increasing sequence of smooth domains
	$U=U_0\subset U_1\subset U_2\subset\cdots\subset M$ such that each $U_i$
	contains the distinguished end $\mathcal E$,
	$\overline{U_i\setminus\mathcal E}$ is compact in $M$, and
	$M=\bigcup_{i=0}^{\infty}U_i$.
	Thus, for each $i$, all ends other than $\mathcal E$ are truncated by $\partial U_i\cap \mathring{M}$. 
	
	Fix $i\ge0$, and consider the following boundary value problem: 
	\begin{equation}\label{Eq: v_i}
		\left\{
		\begin{array}{ccc}
			\Delta_g v_i-fv_i=f &\text{in}&   U_{i},\\
			\frac{\partial v_i}{\partial\mathbf{n} }+ \tilde{f}  v_i= -\tilde{f} &\text{on}&\partial M \cap U_i 
			\\ \frac{\partial v_i}{\partial  \mathbf n'_i}=0&\text{on}&\partial U_i\cap \mathring{M}, 
		\end{array}
		\right.
	\end{equation} 
	where $\mathbf n'_i$ denotes the outward unit normal to the domain $U_i$ along
	$\partial U_i\cap \mathring{M}$. To solve this problem on the unbounded domain
	$U_i$, for all sufficiently large $R$ consider the truncation
	$U_{i,R}=(U_i\setminus\mathcal E)\cup\{x\in\mathcal E:r(x)\le R\}$,
	so that
		\begin{align*}
			\partial U_{i,R}
			&=(\partial M\cap U_{i,R})\cup(\partial U_i\cap \mathring{M})
			\cup\partial B_R,\\
			\partial B_R
			&=\{x\in\mathcal E:r(x)=R\},
			\qquad\text{a coordinate hemisphere}.
		\end{align*}
	On $U_{i,R}$, we consider the following approximating problem:
	\begin{equation}\label{Eq: U_iR}
		\left\{
		\begin{array}{ccc}
			\Delta_g v_{i,R}-fv_{i,R}=f &\text{in}&   U_{i,R},\\
			\frac{\partial v_{i,R}}{\partial\mathbf{n} }+ \tilde{f}  v_{i,R}=  -\tilde{f} &\text{on}&\partial M \cap U_{i,R} 
			\\ \frac{\partial v_{i,R}}{\partial  \mathbf n'_i}=0&\text{on}&\partial U_i\cap \mathring{M}\\ 
			v_{i,R}=0&\text{on}&\partial B_R.
		\end{array}
		\right.
	\end{equation}
	We solve \eqref{Eq: U_iR} weakly in the closed subspace of
	$H^1(U_{i,R})$ consisting of functions with zero trace on $\partial B_R$.
	Extending such functions by zero across $\partial B_R$,
	\eqref{Eq: Sobolev}, \eqref{Eq: Sobolev trace}, H\"older's inequality, and
	\eqref{Eq: f slight negative} give
	\begin{align*}
		&\int_{U_{i,R}}|\nabla_g\phi|^2\,\mathrm d\mu_g
		+\int_{U_{i,R}}f\phi^2\,\mathrm d\mu_g
		+\int_{\partial M\cap U_{i,R}}\tilde f\phi^2\,\mathrm d\sigma_g\\
		&\quad\geq\int_{U_{i,R}}|\nabla_g\phi|^2\,\mathrm d\mu_g
		-\int_Uf_-\phi^2\,\mathrm d\mu_g
		-\int_{\partial M\cap U}\tilde f_-\phi^2\,\mathrm d\sigma_g\\
		&\geq\int_{U_{i,R}}|\nabla_g\phi|^2\,\mathrm d\mu_g
		-\frac{c_S}{4}\left(\int_U|\phi|^{\frac{2n}{n-2}}\,\mathrm d\mu_g\right)^{\frac{n-2}{n}}
		-\frac{c_T}{4}\left(\int_{\partial M\cap U}|\phi|^{\frac{2(n-1)}{n-2}}\,\mathrm d\sigma_g\right)^{\frac{n-2}{n-1}}\\
		&\geq\frac12\int_{U_{i,R}}|\nabla_g\phi|^2\,\mathrm d\mu_g.
	\end{align*}
	Since the trace of $\phi$ vanishes on $\partial B_R$, the Poincar\'e
	inequality shows that the associated bilinear form is coercive. Moreover, the
	linear functional
	\[
	\phi\longmapsto-\int_{U_{i,R}}f\phi\,\mathrm d\mu_g
	-\int_{\partial M\cap U_{i,R}}\tilde f\phi\,\mathrm d\sigma_g
	\]
	is bounded on this space by H\"older's inequality and the same Sobolev and
	trace inequalities. The Lax--Milgram theorem therefore gives a unique weak
	solution $v_{i,R}$ of \eqref{Eq: U_iR}. Standard regularity for mixed boundary
	value problems gives the required $C^{2,\alpha}$ regularity away from the
	interfaces of the boundary conditions.

	Testing \eqref{Eq: U_iR} with $v_{i,R}$ and using the boundary conditions,
	we obtain
	\begin{align*}
		\int_{ U_{i,R}}|\nabla_g v_{i,R}|^2\,\mathrm d\mu_g
		=&-\int_{  U_{i,R}}f v_{i,R}^2\,\mathrm d\mu_g
		-\int_{ U_{i,R}}f v_{i,R}\,\mathrm d\mu_g \\
		&-\int_{\partial M\cap U_{i,R}}\tilde f v_{i,R}^2\,\mathrm d\sigma_g
		-\int_{\partial M\cap U_{i,R}}\tilde f v_{i,R}\,\mathrm d\sigma_g .
	\end{align*} 
	Moreover, \eqref{Eq: Sobolev}, \eqref{Eq: Sobolev trace}, and H\"older's
	inequality give
	\begin{align*}
		&\frac{c_S}{2}\left(\int_U|v_{i,R}|^{\frac{2n}{n-2}}\,\mathrm d\mu_g\right)^{\frac{n-2}{n}}
		+\frac{c_T}{2}\left(\int_{\partial M\cap U}|v_{i,R}|^{\frac{2(n-1)}{n-2}}\,\mathrm d\sigma_g\right)^{\frac{n-2}{n-1}}\\
		&\quad\leq\int_{U_{i,R}}|\nabla_gv_{i,R}|^2\,\mathrm d\mu_g\\
		&\quad\leq\left(\int_U|f_-|^{\frac n2}\,\mathrm d\mu_g\right)^{\frac2n}
		\left(\int_U|v_{i,R}|^{\frac{2n}{n-2}}\,\mathrm d\mu_g\right)^{\frac{n-2}{n}}\\
		&\qquad+\left(\int_U|f|^{\frac{2n}{n+2}}\,\mathrm d\mu_g\right)^{\frac{n+2}{2n}}
		\left(\int_U|v_{i,R}|^{\frac{2n}{n-2}}\,\mathrm d\mu_g\right)^{\frac{n-2}{2n}}\\
		&\qquad+\left(\int_{\partial M\cap U}|\tilde f_-|^{n-1}\,\mathrm d\sigma_g\right)^{\frac1{n-1}}
		\left(\int_{\partial M\cap U}|v_{i,R}|^{\frac{2(n-1)}{n-2}}\,\mathrm d\sigma_g\right)^{\frac{n-2}{n-1}}\\
		&\qquad+\left(\int_{\partial M\cap U}|\tilde f|^{\frac{2(n-1)}n}\,\mathrm d\sigma_g\right)^{\frac{n}{2(n-1)}}
		\left(\int_{\partial M\cap U}|v_{i,R}|^{\frac{2(n-1)}{n-2}}\,\mathrm d\sigma_g\right)^{\frac{n-2}{2(n-1)}}\\
		&\quad\leq\frac{c_S}{4}\left(\int_U|v_{i,R}|^{\frac{2n}{n-2}}\,\mathrm d\mu_g\right)^{\frac{n-2}{n}}
		+C_f\left(\int_U|v_{i,R}|^{\frac{2n}{n-2}}\,\mathrm d\mu_g\right)^{\frac{n-2}{2n}}\\
		&\qquad+\frac{c_T}{4}\left(\int_{\partial M\cap U}|v_{i,R}|^{\frac{2(n-1)}{n-2}}\,\mathrm d\sigma_g\right)^{\frac{n-2}{n-1}}
		+C_{\tilde f}\left(\int_{\partial M\cap U}|v_{i,R}|^{\frac{2(n-1)}{n-2}}\,\mathrm d\sigma_g\right)^{\frac{n-2}{2(n-1)}},
	\end{align*}
	where
	\[
	C_f:=\left(\int_U|f|^{\frac{2n}{n+2}}\,\mathrm d\mu_g\right)^{\frac{n+2}{2n}},
	\qquad
	C_{\tilde f}:=\left(\int_{\partial M\cap U}|\tilde f|^{\frac{2(n-1)}n}\,\mathrm d\sigma_g\right)^{\frac{n}{2(n-1)}}.
	\]
	Applying Young's inequality, we obtain
	\begin{align*}
		&c_S\left(\int_U|v_{i,R}|^{\frac{2n}{n-2}}\,\mathrm d\mu_g\right)^{\frac{n-2}{n}}
		+c_T\left(\int_{\partial M\cap U}|v_{i,R}|^{\frac{2(n-1)}{n-2}}\,\mathrm d\sigma_g\right)^{\frac{n-2}{n-1}}\\
		&\qquad\leq\frac{16C_f^2}{c_S}+\frac{16C_{\tilde f}^2}{c_T}.
	\end{align*}
	Returning to the preceding energy identity and using this estimate, we also obtain
	\begin{equation}\label{Eq: uniform energy viR}
		\int_{U_{i,R}}|\nabla_gv_{i,R}|^2\,\mathrm d\mu_g\leq C,
	\end{equation}
	where $C$ is independent of $i$ and $R$.
	For each fixed $i$, the preceding bounds and an anchored Poincar\'e inequality give uniform local $L^2$ bounds on $U_i$. Take an increasing sequence $\{R_j\}$ with $R_j\to\infty$. The interior and boundary estimates for linear elliptic equations
	(see \cite[Theorems 5.6.2--5.6.3]{Morrey66} and
	\cite[Theorem 6.3.9]{Morrey66}, respectively) imply that  $\{v_{i,R_j}\}$ is bounded in $C^{2,\alpha}_{\mathrm{loc}}(U_i)$. Hence, up to a subsequence,  we have $ v_{i,R_j}\to v_i$  in $C^2$ on compact subsets of $U_i$. The limit function $v_i$ solves \eqref{Eq: v_i} and satisfies the preceding critical Sobolev bound uniformly in $i$. Fatou's lemma and \eqref{Eq: uniform energy viR} also give
	\[
	\int_{U_i}|\nabla_gv_i|^2\,\mathrm d\mu_g\leq C.
	\]
	We next upgrade these integral estimates to a bound uniform in $i$. On fixed
	compact subsets of $U$, the preceding $L^{\frac{2n}{n-2}}$ and energy bounds,
	together with the local interior and Robin boundary estimates, give a uniform
	$L^\infty$ bound. On the asymptotically flat part of $U$, the same conclusion
	follows from the scale-invariant
	$L^{\frac{2n}{n-2}}$-to-$L^\infty$ estimate on annuli, using reflection near
	$\partial M$. Applying the anchored Poincar\'e inequality and the local elliptic
	estimates on a fixed two-sided collar of $\partial U\cap\mathring M$, we obtain
	\[
	\sup_{\partial U\cap\mathring M}|v_i|\leq C.
	\]
	Since $v_i$ is harmonic on $U_i\setminus U$ and satisfies homogeneous Neumann conditions on the remaining boundary pieces, the mixed-boundary maximum principle yields
	\begin{equation}\label{Eq: L infty bound vi}
		\sup_{ U_i}|v_i|\leq C_0
	\end{equation}
	for some constant $C_0$ independent of $i$. Moreover, the annular estimate and the $L^{\frac{2n}{n-2}}$ bound on the distinguished end imply
	\begin{equation}\label{Eq: 0 limit vi}
		v_i(x)\to0\quad\text{as }r(x)\to\infty\text{ along }\mathcal E.
	\end{equation}
	Since $\Delta_gv_i=0$ in $U_i\setminus U$, the divergence theorem gives
	\begin{align}\label{Eq: v_i p1}
		\int_{\partial U\cap \mathring{M}}\frac{\partial v_i}{\partial\mathbf n'}\,\mathrm d\sigma_g
		&=\int_{\partial U_i\cap \mathring{M}}\frac{\partial v_i}{\partial\mathbf n_i'}\,\mathrm d\sigma_g
		+\int_{\partial M\cap(U_i\setminus U)}\frac{\partial v_i}{\partial\mathbf n}\,\mathrm d\sigma_g
		=0.
	\end{align}
	Here we have used the boundary conditions together with the fact that $f$ and $\tilde f$ are supported in $U$. Similarly, multiplying $\Delta_gv_i=0$ by $v_i$ and integrating by parts over $U_i\setminus U$ yields
	\begin{equation}\label{Eq: v_i p2}
		\int_{\partial U\cap \mathring{M}}v_i\frac{\partial v_i}{\partial\mathbf n'}\,\mathrm d\sigma_g
		=-\int_{ U_i\setminus U }|\nabla_gv_i|^2\,\mathrm d\mu_g\leq0.
	\end{equation}

	Returning to the original unknown, set $u_i=1+v_i$. Then $u_i$ satisfies
	\begin{equation*}
		\left\{
		\begin{array}{ccc}
			\Delta_g u_i-fu_i=0 &\text{in}&  U_i,\\
			\frac{\partial u_i}{\partial\mathbf n}+\tilde f u_i=0 &\text{on}& \partial M\cap U_i,\\
			\frac{\partial u_i}{\partial\mathbf n_i'}=0 &\text{on}& \partial U_i\cap \mathring{M}.
		\end{array}
		\right.
	\end{equation*}
	We first show that $u_i>0$ in $ U_i$. By \eqref{Eq: 0 limit vi}, $u_i\to1$ as $r(x)\to\infty$ along $\mathcal E$. It therefore suffices to prove that $u_i\geq0$, since the Harnack inequality \cite[Theorem 8.20]{GT2001} then yields the strict positivity of $u_i$. Suppose, to the contrary, that $\Omega_{i,-}=\{x\in U_i:u_i(x)<0\}$ is nonempty. By the construction of $U_i$, the distinguished end $\mathcal E$ is the only noncompact end of $U_i$. Since $u_i>0$ near infinity along $\mathcal E$, the set $\Omega_{i,-}$ is relatively compact in $U_i$. Testing the weak equation with the negative part of $u_i$, and using the boundary conditions, we obtain 
	\begin{equation*}
		\int_{\Omega_{i,-}}|\nabla_g u_i|^2\,\mathrm d\mu_g
		=-\int_{\Omega_{i,-}}fu_i^2\,\mathrm d\mu_g-\int_{\partial M\cap\Omega_{i,-}}\tilde f\,u_i^2\,\mathrm d\sigma_g.
	\end{equation*}
	By H\"older's inequality, \eqref{Eq: Sobolev}, \eqref{Eq: Sobolev trace}, and \eqref{Eq: f slight negative}, we have
	\begin{align*}
		\int_{\Omega_{i,-}}|\nabla_g u_i|^2\,\mathrm d\mu_g
		&\leq \left(\int_{U }|f_-|^{\frac n2}\,\mathrm d\mu_g\right)^{\frac2n}
		\left(\int_{\Omega_{i,-}\cap U}|u_i|^{\frac{2n}{n-2}}\,\mathrm d\mu_g\right)^{\frac{n-2}{n}}\\
		&\quad+\left(\int_{\partial M \cap U}|\tilde f_-|^{n-1}\,
		\mathrm d\sigma_g\right)^{\frac1{n-1}} \left(\int_{\partial M\cap\Omega_{i,-}\cap U}
		|u_i|^{\frac{2(n-1)}{n-2}}\,\mathrm d\sigma_g
		\right)^{\frac{n-2}{n-1}}\\
		&\leq \frac12\int_{\Omega_{i,-}}|\nabla_g u_i|^2\,\mathrm d\mu_g.
	\end{align*}
	Hence $|\nabla_g u_i|\equiv0$ in $\Omega_{i,-}$. Since each connected component of $\Omega_{i,-}$ has nonempty boundary on which $u_i=0$, this is impossible. Thus $u_i\geq0$ in $U_i$. The Harnack inequality \cite[Theorem 8.20]{GT2001}, together with the boundary point lemma and the Robin boundary condition, gives $u_i>0$.

	The preceding bounds allow us to pass to the limit as $i\to\infty$ and
	construct the desired solution of \eqref{system u}. By
	\eqref{Eq: L infty bound vi}, the interior and boundary Schauder estimates,
	and a diagonal argument, after passing to a subsequence we have $u_i\to u$ in
	$C^2$ on compact subsets of $M$ for some nonnegative function $u$ satisfying
	\[
	\Delta_g u-fu=0\quad\text{in }M,\qquad
	\frac{\partial u}{\partial\mathbf n}+\tilde f u=0\quad\text{on }\partial M.
	\]
	We first verify the normalization at $\mathcal E$ and the global positive
	lower bound.
	
	Set $v=u-1$. Passing to the limit in the preceding energy estimate gives  
	\begin{equation}\label{Eq: integral v}
		c_S\left(\int_{  U}|v|^{\frac{2n}{n-2}}\,\mathrm d\mu_g\right)^{\frac{n-2}{n}}
		+c_T\left(\int_{\partial M\cap U}|v|^{\frac{2(n-1)}{n-2}}\,\mathrm d\sigma_g\right)^{\frac{n-2}{n-1}}
		\leq C.
	\end{equation}  
	Applying the local $L^p$-to-$L^\infty$ estimate \cite[Theorem 8.17]{GT2001} on the compact part of $U$ and on asymptotic annuli, together with the standard reflection argument near $\partial M$, we obtain
	\[
	\sup_U|v|\le C
	\qquad\text{and}\qquad
	v(x)\to0\quad\text{as }r(x)\to\infty\text{ along }\mathcal E.
	\]  
	Consequently, $u\to 1$ at infinity along $\mathcal E$. In particular, $u\not\equiv0$, and the Harnack inequality \cite[Theorem 8.20]{GT2001}, together with the boundary point lemma, yields $u>0$ on $M$.
	Since $u\to1$ along $\mathcal E$ and $\overline{U\setminus\mathcal E}$ is
	compact, positivity gives $\inf_Uu>0$. The local convergence $u_i\to u$
	on the compact interface $\partial U\cap\mathring M$ therefore gives a
	constant $c>0$, independent of all sufficiently large $i$, such that
	$u_i\ge c$ there. On $U_i\setminus U$, the function $u_i$ is harmonic and
	satisfies homogeneous Neumann conditions on all boundary pieces other than
	that interface. The mixed-boundary minimum principle gives $u_i\ge c$ on
	$U_i\setminus U$. Passing to the limit on compact subsets yields $u\ge c$
	on $M\setminus U$. Thus $\inf_Mu>0$.
	
	We then determine the asymptotic behavior and flux of $u$. Outside a
	compact subset of $\mathcal E$, the function $v=u-1$ is $g$-harmonic and
	satisfies the homogeneous conormal condition on the noncompact boundary. In
	the coordinates of Definition~\ref{Def A.F.}, these equations take the form
	\begin{align*}
	\Delta_0v&=-\partial_i\bigl(O_2(r^{2-n})\partial_jv\bigr)
	&&\text{in }\mathbb R^n_+\setminus\overline{\mathbb B}^+_R,\\
	\partial_{x_n}v&=O_2(r^{2-n})|\partial v|
	&&\text{on }\partial\mathbb R^n_+\setminus\mathbb B_R.
	\end{align*}
	Use the Euclidean Neumann kernel on $\mathbb R^n_+$, namely the sum of the fundamental solution and its reflection. The annular estimates and \eqref{Eq: integral v} give the initial decay required in the potential-theoretic argument of \cite[Lemma 3.2]{SY79}. Applying the corresponding Neumann Green representation and iterating its standard weighted estimates first gives $|v|+r|\partial v|+r^2|\partial^2v|\leq Cr^{2-n}$. Substituting this estimate back into the Green representation and using the coefficient decay $O_2(r^{2-n})$ separates the monopole term from the integrable error and gives $v=Ar^{2-n}+\omega$ along $\mathcal E$, where $A$ is a constant and $|\omega|+r|\partial\omega|+r^2|\partial^2\omega|\leq Cr^{1-n}$.
	Finally, since $u_i=1+v_i$, using \eqref{Eq: v_i p1} and \eqref{Eq: v_i p2} and passing to the limit using the local $C^2$ convergence of $u_i$ gives 
	\[
	\int_{\partial U\cap \mathring{M}}\frac{\partial u}{\partial\mathbf n'}\,\mathrm d\sigma_g=0,
	\qquad
	\int_{\partial U\cap \mathring{M}}u\frac{\partial u}{\partial\mathbf n'}\,\mathrm d\sigma_g\leq0.
	\]
	To identify the coefficient $A$, for $R$ sufficiently large let
	$U_R=(U\setminus\mathcal E)\cup\{x\in\mathcal E:r(x)<R\}$,
	where $r$ is the asymptotically flat coordinate on the distinguished end $\mathcal E$, and set $S_R=\{r=R\}\cap\mathcal E$. Integrating $\Delta_g u=fu$ over $U_R$ and using the boundary condition in \eqref{system u}, we obtain
	\[
	\int_{U_R}fu\,\mathrm d\mu_g
	=
	\int_{S_R}\frac{\partial u}{\partial\nu}\,\mathrm d\sigma_g
	-\int_{\partial M\cap U_R}\tilde f u\,\mathrm d\sigma_g
	+\int_{\partial U\cap\mathring M}\frac{\partial u}{\partial\mathbf n'}\,\mathrm d\sigma_g.
	\]
	By the zero-flux identity proved above, the last term vanishes. Since
	$u=1+Ar^{2-n}+O(r^{1-n})$, we have
	\[
	\lim_{R\to\infty}
	\int_{S_R}\frac{\partial u}{\partial\nu}\,\mathrm d\sigma_g
	=
	\frac{2-n}{2}|\mathbb S^{n-1}|A.
	\]
	Letting $R\to\infty$ therefore gives
	\[
	\frac{2-n}{2}|\mathbb S^{n-1}|A
	=
	\int_Ufu\,\mathrm d\mu_g
	+
	\int_{\partial M\cap U}\tilde f u\,\mathrm d\sigma_g,
	\]
	which gives the asserted formula for $A$. Thus $u$ has all the required properties.
\end{proof}

The solvability result above yields the density deformation needed before
doubling. The following proposition is an analogue of
\cite[Proposition 4.1]{ABL16}, adapted to a distinguished asymptotically flat
end with arbitrary remaining ends.
\begin{prop}\label{conf:flat:assump}
	Let $(M,g,\mathcal E)$ be a manifold with arbitrary ends and distinguished asymptotically flat end $\mathcal E$. Assume that
	$R_g\ge0$ and $H_g\ge0$.
	Then, for every $\varepsilon>0$, there exists a metric $\bar g$ such that:
	\begin{itemize}
		\item[(i)] $R_{\bar g}\ge0$ and $H_{\bar g}\ge0$ on $M$, while
		$R_{\bar g}\equiv0$ and $H_{\bar g}\equiv0$ near infinity along
		$\mathcal E$;
		\item[(ii)] $\bar g$ is conformally flat near infinity along $\mathcal E$;
		\item[(iii)]
		\[
		\left|\mathfrak m(M,\bar g,\mathcal E)-\mathfrak m(M,g,\mathcal E)\right|<\varepsilon.
		\]
	\end{itemize}
	Moreover, $\bar g=u^{\frac{4}{n-2}}\delta$ near infinity along
	$\mathcal E$, where $u>0$ and $u(x)\to1$ as $r(x)\to\infty$.
\end{prop}

\begin{proof}
	We first cut off the metric to the Euclidean metric near infinity and then
	repair the resulting scalar and boundary mean curvatures conformally. Take a
	fixed nonnegative cutoff function $\zeta:\mathbb R\to[0,1]$ such that
	$\zeta\equiv1$ in $(-\infty,2]$ and $\zeta\equiv0$ in $[3,+\infty)$. For
	$s>1$, set $\zeta_s(x)=\zeta(r(x)/s)$. On the distinguished end $\mathcal E$, define
	$g_s=\zeta_sg+(1-\zeta_s)\delta$,
	and set $g_s=g$ away from the asymptotic coordinate neighborhood of $\mathcal E$. Then
	\[
	|g_s-\delta|+r|\partial g_s|+r^2|\partial^2g_s|\le Cr^{2-n},
	\]
	and the metrics $g_s$ are uniformly equivalent to $g$ in a fixed neighborhood $U$ of $\mathcal E$.
	After decreasing the constants if necessary, the Sobolev and trace inequalities
	of Lemma~\ref{Lem: Sobolev} hold for all sufficiently large $s$, with constants
	$c_S$ and $c_T$ independent of $s$.
	
	We repair the curvature created in the transition annulus by seeking a
	positive solution $u_s$ of
	\[
	\left\{
	\begin{aligned}
		-a_n\Delta_{g_s}u_s+R_{g_s}u_s&=\zeta_sR_gu_s,&&\text{in }M,\\
		b_n\frac{\partial u_s}{\partial\mathbf n_{g_s}}+H_{g_s}u_s&=\zeta_sH_gu_s,&&\text{on }\partial M.
	\end{aligned}
	\right.
	\]
	Equivalently,
	\begin{equation}\label{eq:us}
		\left\{
		\begin{aligned}
			\Delta_{g_s}u_s-f_su_s&=0,&&\text{in }M,\\
			\frac{\partial u_s}{\partial\mathbf n_{g_s}}+\tilde f_su_s&=0,&&\text{on }\partial M,
		\end{aligned}
		\right.
	\end{equation}
	where $f_s=a_n^{-1}(R_{g_s}-\zeta_sR_g)$ and
	$\tilde f_s=b_n^{-1}(H_{g_s}-\zeta_sH_g)$.
	
	The functions $f_s$ and $\tilde f_s$ are supported in the transition region $\{2s\le r\le3s\}$ of the distinguished end. By the asymptotic decay of $g$ and the definition of $g_s$, we have
	\[
	|f_s|\le Cs^{-n},\qquad |\tilde f_s|\le Cs^{1-n}
	\]
	on their respective supports. Hence, for all sufficiently large $s$,
	\[
	\left(\int_U|f_{s,-}|^{\frac n2}\,d\mu_{g_s}\right)^{\frac2n}\le\frac{c_S}{4},
	\qquad
	\left(\int_{\partial M\cap U}|\tilde f_{s,-}|^{n-1}\,d\sigma_{g_s}\right)^{\frac1{n-1}}\le\frac{c_T}{4}.
	\]
	By Proposition~\ref{Prop: the conformal factor}, \eqref{eq:us} admits a positive solution $u_s$ for all sufficiently large $s$.
	
	Set $\bar g_s=u_s^{\frac4{n-2}}g_s$.
	The metric $g_s$ is complete, and Proposition~\ref{Prop: the conformal factor}
	gives $\inf_Mu_s>0$. Hence $\bar g_s$ is complete.
	By the conformal transformation laws and \eqref{eq:us},
	\[
	R_{\bar g_s}=\zeta_sR_gu_s^{-\frac4{n-2}}\ge0,\qquad H_{\bar g_s}=\zeta_sH_gu_s^{-\frac2{n-2}}\ge0.
	\]
	Since $\zeta_s=0$ and $g_s=\delta$ for $r\ge3s$,
	\[
	R_{\bar g_s}=H_{\bar g_s}=0,\qquad \bar g_s=u_s^{\frac4{n-2}}\delta
	\]
	near infinity along $\mathcal E$, with $u_s\to1$ there.
	
	It remains to verify that the deformation changes the mass by an arbitrarily
	small amount. Let
	\[
	\mathcal A_s=\{x\in\mathcal E:2s\le r(x)\le3s\}.
	\]
	The support of $f_s$ and $\tilde f_s$ is contained in $\mathcal A_s$. Since
	\[
	\|f_s\|_{L^{n/2}(\mathcal A_s,g_s)}
	+\|\tilde f_s\|_{L^{n-1}(\partial M\cap\mathcal A_s,g_s)}
	\to0,
	\]
	the energy estimate in Proposition~\ref{Prop: the conformal factor} gives that $u_s-1$ tends to zero in the corresponding critical Sobolev norms on $\mathcal A_s$.
	
	After rescaling $\mathcal A_s$ to the fixed annulus $\{2\le r\le3\}$, the standard interior and boundary $L^p$--to--$L^\infty$ estimates give
	\begin{equation}\label{us transition}
		\sup_{\mathcal A_s}|u_s-1|=o(1).
	\end{equation}
	
	By the asymptotic expansion in Proposition~\ref{Prop: the conformal factor},
	\[
	u_s(x)=1+A_s|x|^{2-n}+O(|x|^{1-n})
	\]
	near infinity along $\mathcal E$, where
	\begin{equation}\label{As approximation}
		A_s=-\frac{2}{(n-2)|\mathbb S^{n-1}|}
		\left(\int_M f_su_s\,d\mu_{g_s}+\int_{\partial M}\tilde f_su_s\,d\sigma_{g_s}\right).
	\end{equation}
	
	Using \eqref{us transition}, the support properties of $f_s$ and $\tilde f_s$, and the
	uniform $L^1$ bounds on the transition region, we obtain
	\begin{align*}
	&\int_M f_su_s\,d\mu_{g_s}
	+\int_{\partial M}\tilde f_su_s\,d\sigma_{g_s}\\
	&\quad=\frac1{a_n}\biggl\{
	\int_M(R_{g_s}-\zeta_sR_g)\,d\mu_{g_s} +2\int_{\partial M}(H_{g_s}-\zeta_sH_g)\,d\sigma_{g_s}
	\biggr\}+o(1).
	\end{align*}
	
	On the asymptotic region, the standard expansions
	\begin{align*}
	R_{g_s}&=\bigl((g_s)_{ij,j}-(g_s)_{jj,i}\bigr)_{,i}+O(r^{2-2n}),\\
	H_{g_s}&=\frac12\left[-\bigl((g_s)_{ij,j}-(g_s)_{jj,i}\bigr)\mathbf n^i
	+(g_s)_{n\alpha,\alpha}\right]+O(r^{3-2n}),
	\end{align*}
	together with the divergence theorem on $\mathcal A_s$, yield
	\begin{align*}
	&\int_M(R_{g_s}-\zeta_sR_g)\,d\mu_{g_s}
	+2\int_{\partial M}(H_{g_s}-\zeta_sH_g)\,d\sigma_{g_s}\\
	&\qquad=-2(n-1)\omega_{n-1}\mathfrak m(M,g,\mathcal E)+o(1).
	\end{align*}
	
	Indeed, the outer flux vanishes since $g_s=\delta$ for $r\ge3s$, whereas the inner flux converges to $2(n-1)\omega_{n-1}\mathfrak m(M,g,\mathcal E)$. The error terms are $O(s^{2-n})$, while the terms involving $(1-\zeta_s)R_g$ and $(1-\zeta_s)H_g$ tend to zero because $R_g\in L^1(\mathcal E)$ and $H_g\in L^1(\partial M\cap\mathcal E)$. Therefore, since $a_n=4(n-1)/(n-2)$, \eqref{As approximation} gives $A_s=\mathfrak m(M,g,\mathcal E)+o(1)$.
	
	Since $\bar g_s=u_s^{\frac4{n-2}}\delta$ for $r\ge3s$, a direct computation from Definition~\ref{Def mass} gives $\mathfrak m(M,\bar g_s,\mathcal E)=A_s$.
	Consequently,
	\[
	\mathfrak m(M,\bar g_s,\mathcal E)
	\longrightarrow
	\mathfrak m(M,g,\mathcal E)
	\qquad\text{as }s\to\infty.
	\]
	Choosing $s$ sufficiently large and setting $\bar g=\bar g_s$ proves (iii), and the proof is complete.
\end{proof}

\subsection{Doubling, local smoothing, and conformal deformation}

We now apply the preceding density deformation and pass from the boundary
problem to a boundaryless one. Fix a metric $\bar g$ supplied by
Proposition~\ref{conf:flat:assump}.

\subsubsection{Doubling along the boundary}

Let $\widetilde M=(M\times\{-1,1\})/\!\sim$ be the smooth double of $M$
across $\partial M$, where the smooth structure is
the one induced by a smooth collar of $\partial M$. Let
$\pi:\widetilde M\to M$ be the natural projection and set
$\mathcal S=\pi^{-1}(\partial M)$.
Thus $\mathcal S$ is the doubling hypersurface. Notice that no uniform
metric tubular radius along the noncompact hypersurface $\mathcal S$ is being
asserted here.

Define $\widetilde g=\pi^*\bar g$. In local Gaussian coordinates on either
side of a point of $\mathcal S$, the metric is of the form
$dt^2+\gamma_{\bar g}(|t|)$.
Consequently, $\widetilde g$ is continuous on $\widetilde M$ and smooth away
from $\mathcal S$. Let $\widetilde{\mathcal E}$ denote the double of the
distinguished end $\mathcal E$.

\begin{lem}
	The doubled metric $\widetilde g$ is $C^2$-asymptotically flat along
	$\widetilde{\mathcal E}$, and
	\[
	\mathfrak m(\widetilde M,\widetilde g,\widetilde{\mathcal E})
	=2\mathfrak m(M,\bar g,\mathcal E).
	\]
\end{lem}

\begin{proof}
	By Proposition~\ref{conf:flat:assump}, near infinity along $\mathcal E$ we
	have $\bar g=u^{\frac4{n-2}}\delta$, $H_{\bar g}=0$, and $u\to1$.
	The conformal transformation law for the mean curvature gives
	$\partial_{\nu_\delta}u=0$ on $\partial M$
	near infinity. Hence $u$ extends evenly across the boundary hyperplane to a
	$C^2$ function $\widetilde u$, and
	$\widetilde g=\widetilde u^{\frac4{n-2}}\delta$
	near infinity along $\widetilde{\mathcal E}$. In particular,
	$\widetilde g$ is already $C^2$ there.
	
	Since $\bar g_{\alpha n}=0$ near infinity, the boundary contribution in
	Definition~\ref{Def mass} vanishes. Reflection symmetry then gives
	\begin{align*}
		\mathfrak m(\widetilde M,\widetilde g,\widetilde{\mathcal E})
		&=\frac{1}{2(n-1)\omega_{n-1}}\lim_{r\to\infty}
		\int_{\mathbb S_r^{n-1}}
		\bigl((\widetilde g)_{ij,j}-(\widetilde g)_{jj,i}\bigr)
		\nu^i\,d\mathbb S_r^{n-1}\\
		&=\frac{2}{2(n-1)\omega_{n-1}}\lim_{r\to\infty}
		\int_{\mathbb S_{r,+}^{n-1}}
		\bigl((\bar g)_{ij,j}-(\bar g)_{jj,i}\bigr)
		\nu^i\,d\mathbb S_{r,+}^{n-1}\\
		&=2\mathfrak m(M,\bar g,\mathcal E).
	\end{align*}
\end{proof}

\subsubsection{Local smoothing}

We next describe a smoothing construction that requires geometric control
only in a prescribed neighborhood of the distinguished end. Let
$\widetilde U\subset\widetilde V\subset\widetilde W$ be neighborhoods of
$\widetilde{\mathcal E}$ such that the closures of the
successive differences, as well as $\overline{\widetilde W\setminus\widetilde{\mathcal E}}$, are compact.
Choose the neighborhoods with smooth boundaries.

The doubled metric is already $C^2$ sufficiently far out along
$\widetilde{\mathcal E}$. Therefore, the portion of $\mathcal S\cap
\widetilde W$ on which $\widetilde g$ is not $C^2$ has compact closure. A
finite collection of Gaussian normal charts covers this portion, and the
normal smoothing of Miao \cite[Proposition 3.1]{Miao03} can be performed there
with one parameter $\delta>0$. Denote the resulting $C^2$ metric on
$\widetilde W$ by $h_\delta$. It is chosen to agree with $\widetilde g$
near infinity along $\widetilde{\mathcal E}$ and away from a compact collar
of the nonsmooth part of $\mathcal S\cap\widetilde W$. Moreover,
$h_\delta\to\widetilde g$ uniformly on $\widetilde V$.

To extend this local smoothing, choose an arbitrary complete smooth metric
$h$ on the smooth manifold $\widetilde M$. Let
$\beta\in C^\infty(\widetilde M)$ satisfy
\[
0\le\beta\le1,
\qquad
\beta\equiv1\text{ on }\widetilde V,
\qquad
\operatorname{supp}\beta\subset\widetilde W.
\]
On $\widetilde W$ set $\widetilde g_\delta=\beta h_\delta+(1-\beta)h$, and
set $\widetilde g_\delta=h$ outside $\widetilde W$. Since the cone of
positive-definite symmetric tensors is convex, $\widetilde g_\delta$ is a
global $C^2$ Riemannian metric. It is complete: it agrees with the complete
asymptotically flat metric $\widetilde g$ sufficiently far out along
$\widetilde{\mathcal E}$, agrees with the complete metric $h$ sufficiently
far out along every other end, and differs from these metrics only in compact
transition regions. Moreover,
\[
\widetilde g_\delta=\widetilde g
\quad\text{near infinity along }\widetilde{\mathcal E},
\]
and therefore
\begin{equation}\label{mass smoothing}
	\mathfrak m(\widetilde M,\widetilde g_\delta,
	\widetilde{\mathcal E})
	=\mathfrak m(\widetilde M,\widetilde g,
	\widetilde{\mathcal E}).
\end{equation}

It remains to control the scalar curvature introduced by the smoothing. Let $R_\delta$ denote the scalar curvature of $\widetilde g_\delta$.
Because $H_{\bar g}\ge0$, the singular scalar-curvature contribution of the
double across $\mathcal S$ has the favorable sign. Miao's local estimate,
applied on the compact part of $\mathcal S\cap\widetilde V$ where smoothing
is required, gives a compact set $K\Subset\widetilde V$, independent of
$\delta$, such that
\begin{equation}\label{local Miao support}
	\operatorname{supp}(R_\delta)_-\cap\widetilde V\subset K,
	\qquad
	0\le (R_\delta)_-\le C_K,
\end{equation}
and the support in \eqref{local Miao support} has volume $O(\delta)$.
Consequently,
\begin{equation}\label{Rneg small local}
\begin{aligned}
	\|(R_\delta)_-\|_{L^{n/2}(\widetilde V)}
	&=O(\delta^{2/n}),\\
	\|(R_\delta)_-\|_{L^{\frac{2n}{n+2}}(\widetilde V)}
	&=O\!\left(\delta^{\frac{n+2}{2n}}\right),\\
	\|(R_\delta)_-\|_{L^1(\widetilde V)}
	&=O(\delta).
\end{aligned}
\end{equation}
All norms in this subsection are taken with respect to
$\widetilde g_\delta$.

\subsubsection{Conformal correction}

We now remove the small negative part of the scalar curvature in the protected
neighborhood of the distinguished end. Choose
$\chi\in C^\infty(\widetilde M)$ such that
\[
0\le\chi\le1,
\qquad
\chi\equiv1\text{ on }\widetilde U,
\qquad
\operatorname{supp}\chi\subset\widetilde V,
\]
and fix a compact set $K'$ with $K\Subset K'\Subset\widetilde V$. Choose a
smooth nonnegative function $q_\delta$, supported in $K'$, such that
$q_\delta\ge\chi(R_\delta)_-$ and such that $q_\delta$ satisfies the three estimates in
\eqref{Rneg small local}. By \eqref{Rneg small local}, for all sufficiently
small $\delta$, Zhu's boundaryless conformal-factor result
\cite[Proposition 2.2]{Zhu23}, applied with
$f_\delta=-a_n^{-1}q_\delta$, gives a positive solution
\begin{equation}\label{conf u delta}
	\Delta_{\widetilde g_\delta}u_\delta
	+a_n^{-1}q_\delta u_\delta=0,
	\qquad
	u_\delta\to1
	\quad\text{along }\widetilde{\mathcal E}.
\end{equation}
The exhaustion construction and the maximum principle also give
\begin{equation}\label{u delta geq one}
	u_\delta\ge1\quad\text{on }\widetilde M.
\end{equation}

Set $w_\delta=u_\delta-1$. The flux inequalities in
\cite[Proposition 2.2]{Zhu23} and the same energy estimate as in the proof of
Proposition~\ref{Prop: the conformal factor} give
\begin{equation}
	\|w_\delta\|_{L^{\frac{2n}{n-2}}(\widetilde U)}
	\le C\|q_\delta\|_{L^{\frac{2n}{n+2}}(\widetilde V)}=o(1).
\end{equation}
The local $L^p$-to-$L^\infty$ estimate on the fixed compact set $K'$ then
yields
\begin{equation}\label{local uniform}
	\sup_{K'}|u_\delta-1|=o(1).
\end{equation}

Define $\widehat g_\delta=u_\delta^{\frac4{n-2}}\widetilde g_\delta$.
Using $a_n^{-1}=(n-2)/(4(n-1))$ and \eqref{conf u delta}, the conformal
transformation formula gives
\begin{equation}
	R_{\widehat g_\delta}
	=u_\delta^{-\frac4{n-2}}
	\bigl(R_\delta+q_\delta\bigr).
\end{equation}
In particular,
\begin{equation}\label{R hat nonnegative U}
	R_{\widehat g_\delta}\ge0
	\quad\text{on }\widetilde U.
\end{equation}
Moreover, \eqref{u delta geq one} and the completeness of
$\widetilde g_\delta$ imply that $\widehat g_\delta$ is complete.

Finally, we compare the masses before and after the conformal correction. Along
$\widetilde{\mathcal E}$, Proposition 2.2 of \cite{Zhu23} gives
$u_\delta=1+A_\delta r^{2-n}+O_2(r^{1-n})$, where
\begin{equation}
	A_\delta
	=\frac{a_n^{-1}}{(n-2)|\mathbb S^{n-1}|}
	\int_{\widetilde M}q_\delta u_\delta\,
	d\mu_{\widetilde g_\delta}.
\end{equation}
By \eqref{local uniform} and \eqref{Rneg small local},
$0\le A_\delta\le C\delta$, and hence $A_\delta\to0$. The maximum principle on the complement of $K'$
and the harmonic estimates on the fixed asymptotically flat end also give
uniform weighted estimates for $u_\delta-1$ there. A direct computation using the ADM
normalization in Definition~\ref{Def mass} gives
\begin{equation}\label{mass conformal correction}
	\mathfrak m(\widetilde M,\widehat g_\delta,
	\widetilde{\mathcal E})
	=\mathfrak m(\widetilde M,\widetilde g_\delta,
	\widetilde{\mathcal E})
	+2A_\delta.
\end{equation}
Combining \eqref{mass smoothing} and
\eqref{mass conformal correction}, we obtain
\begin{equation}\label{mass corrected convergence}
	\mathfrak m(\widetilde M,\widehat g_\delta,
	\widetilde{\mathcal E})
	\longrightarrow
	\mathfrak m(\widetilde M,\widetilde g,
	\widetilde{\mathcal E}).
\end{equation}
In addition, because all the metrics agree with $\widetilde g$ in their leading
asymptotic coefficients except for the term $A_\delta r^{2-n}$, and
$A_\delta\to0$, their weighted $W^{2,p}_{-q}$ norms on the fixed end
$\widetilde{\mathcal E}$ are uniformly bounded for every $p>n$ and every
$q\in((n-2)/2,n-2)$.

Thus the corrected metrics are complete, have nonnegative scalar curvature on
the prescribed neighborhood $\widetilde U$, and converge to the doubled mass
while retaining uniform control of the distinguished end. These are precisely
the properties needed in the proof of the main theorem.

\section{Proof of Theorem~\ref{Thm: main A}}
\label{sec:proof-main}

\begin{proof}
	
	\textit{Proof of the nonnegativity of the mass.}
	Assume, for contradiction, that $ 	\mathfrak m(M,g,\mathcal E)<0$. 
	By Proposition~\ref{conf:flat:assump}, we may choose $\bar g$ such that $\mathfrak m(M,\bar g,\mathcal E)<0$. 
	Let $(\widetilde M,\widetilde g)$ be its double and set
	\[
	m_0:=\mathfrak m(\widetilde M,\widetilde g,
	\widetilde{\mathcal E})
	=2\mathfrak m(M,\bar g,\mathcal E)<0.
	\]
	
	Fix $p>n$ and $q\in((n-2)/2,n-2)$. With the ADM normalization in
	Definition~\ref{Def mass}, the quantitative dependence in
	\cite[Corollary 1.6]{LLU} implies that there is a constant $D>0$ with the
	following property: any metric in the above asymptotic class whose mass is at
	most $m_0/2<0$ and whose $W^{2,p}_{-q}$ norm on
	$\widetilde{\mathcal E}$ is bounded by a fixed constant must have either a
	point of negative scalar curvature or a point of incompleteness in the
	$D$-neighborhood of $\widetilde{\mathcal E}$.
	
	We now choose the protected neighborhood large enough for this quantitative
	obstruction. Choose nested neighborhoods
	\[
	\widetilde U\subset\widetilde V\subset\widetilde W
	\]
	as in the preceding subsection, with
	\begin{equation}\label{protected distance}
		\operatorname{dist}_{\widetilde g}
		(\widetilde{\mathcal E},\partial\widetilde U)>3D.
	\end{equation}
	Construct $\widetilde g_\delta$ and $\widehat g_\delta$ using these
	neighborhoods. By \eqref{mass corrected convergence} and the uniform
	weighted asymptotic bounds, for all sufficiently small $\delta$,
	\[
	\frac32m_0
	<\mathfrak m(\widetilde M,\widehat g_\delta,
	\widetilde{\mathcal E})
	<\frac12m_0<0,
	\]
	and the same constant $D$ applies to $\widehat g_\delta$.
	
	On $\widetilde U$, the metrics $\widetilde g_\delta$ converge uniformly to
	$\widetilde g$. Together with \eqref{u delta geq one} and
	\eqref{protected distance}, this gives, for all sufficiently small $\delta$,
	\[
	\operatorname{dist}_{\widehat g_\delta}
	(\widetilde{\mathcal E},\partial\widetilde U)>D.
	\]
	Therefore
	\[
	N_D^{\widehat g_\delta}(\widetilde{\mathcal E})
	\subset\widetilde U.
	\]
	By \eqref{R hat nonnegative U},
	\[
	R_{\widehat g_\delta}\ge0
	\quad\text{on }
	N_D^{\widehat g_\delta}(\widetilde{\mathcal E}).
	\]
	On the other hand, $\widehat g_\delta$ is complete by
	\eqref{u delta geq one}. This contradicts
	\cite[Corollary 1.6]{LLU}. Hence
	\[
	\mathfrak m(M,g,\mathcal E)\ge0.
	\]
	
	\textit{Proof of the rigidity.}

	It remains to analyze the equality case. Assume that
	$\mathfrak m(M,g,\mathcal E)=0$.
	
	\textit{Step 1.}   ${\rm Ric}_g=0$ and $A_g=0$.
	
	Let $k$ be an arbitrary compactly supported symmetric $2$-tensor on $M$, with $\operatorname{supp}k=K$, and consider the variation $g_t=g+tk$. Choose a smooth neighborhood $U$ of the distinguished end $\mathcal E$ such that $K\Subset U$ and $\overline{U\setminus\mathcal E}$ is compact. Since $R_{g_t}-R_g$ and $H_{g_t}-H_g$ are supported in $K$ and tend to zero in $C^k$ as $t\to0$, Proposition~\ref{Prop: the conformal factor} provides, for $|t|$ sufficiently small, a positive solution $u_t$ of
	\begin{equation}\label{conf factor ut}
		\left\{
		\begin{aligned}
			-a_n\Delta_{g_t}u_t+(R_{g_t}-R_g) u_t&=0, &&\text{in }M,\\
			b_n\frac{\partial u_t}{\partial\eta_{g_t}}+(H_{g_t}-H_g)u_t
			&=0, &&\text{on }\partial M,\\
			u_t(x)&\to1, &&\text{as }r(x)\to\infty\text{ along }\mathcal E.
		\end{aligned}
		\right.
	\end{equation} 
	Taking $u_t$ to be the solution selected by the exhaustion construction in
	the proof of Proposition~\ref{Prop: the conformal factor}, the corresponding
	estimates give
	\[
	u_t\longrightarrow1
	\quad\text{in }C^{2,\alpha}_{\mathrm{loc}}(M)
	\]
	as $t\to0$. Outside a fixed compact set, $u_t$ is harmonic and satisfies the
	homogeneous Neumann condition. The maximum principle on the exhaustion
	domains, followed by the Harnack inequality on the fixed compact interface,
	therefore gives
	\[
	\inf_Mu_t>0
	\]
	for all sufficiently small $|t|$.
	
	For $|t|$ sufficiently small, define
	\[
	\hat g_t=u_t^{\frac{4}{n-2}}g_t.
	\]
	Then $\hat g_t$ is a complete conformal metric on $M$. By \eqref{conf factor ut} and the conformal transformation laws,
	\[
	R_{\hat g_t}=u_t^{-\frac{4}{n-2}}R_g\ge0,\qquad H_{\hat g_t}=u_t^{-\frac{2}{n-2}}H_g\ge0.
	\]
	Hence $(M,\hat g_t,\mathcal E)$ remains a manifold with arbitrary ends and distinguished asymptotically flat end $\mathcal E$, and therefore $\mathfrak m(M,\hat g_t,\mathcal E)\ge0$. Since $\hat g_0=g$, and the assumption $\mathfrak m(M,g,\mathcal E)=0$ shows that $t=0$ is a local minimum of $t\mapsto\mathfrak m(M,\hat g_t,\mathcal E)$.

	To extract geometric information from this minimizing property, we compute
	the first variation of the mass
	$\mathfrak m(M,\hat g_t,\mathcal E)$. Since $R_{g_t}-R_g$ and
	$H_{g_t}-H_g$ are compactly supported and are of order $O(|t|)$, the
	estimates in the proof of Proposition~\ref{Prop: the conformal factor} give
	\[
	u_t-1=O(|t|)
	\quad\text{in }C^{2,\alpha}_{\mathrm{loc}}(M).
	\]
	Let $t_j\to0$. After passing to a subsequence, we may assume that
	\[
	\frac{u_{t_j}-1}{t_j}\to\dot u
	\quad\text{in }C^2_{\mathrm{loc}}(M).
	\]
	It follows that
	\[
	\frac{\hat g_{t_j}-g}{t_j}
	\to
	\dot{\hat g}:=k+\frac4{n-2}\dot u\,g
	\quad\text{in }C^1_{\mathrm{loc}}(M).
	\]
	Throughout the following calculation, $\delta$ denotes the limit of the
	corresponding difference quotients along the chosen subsequence $t_j$.
	For $r$ sufficiently large, set $U_r=U\cap\{x\in M: r(x)\le r\}$ and $\Sigma_r=\partial M\cap U_r$. Then $U_r$ is compact. Applying the computation in \cite[Proposition 2.1]{ABL16} to $U_r$, 
	and then letting $r\to\infty$, gives   
	\begin{align}
		&2(n-1)\omega_{n-1}\,\delta\mathfrak m(M,g,\mathcal E)
		\\=&\delta\left(\int_U R_g\,\mathrm d\mu_g
		+2\int_{\partial M\cap U}H_g\,\mathrm d\sigma_g\right)\\
		&\quad+\int_U\Big\langle {\rm Ric}_g-\frac12R_gg,
		\dot{\hat g}\Big\rangle\,\mathrm d\mu_g 
		 +\int_{\partial M\cap U}\Big\langle A_g-H_gh_g,
		\dot{\hat g}^{T}\Big\rangle\,\mathrm d\sigma_g\\
		&\quad-\int_{\partial U\cap\mathring M}\mathbf n'{}^i
		\big(\nabla_j\dot{\hat g}^j{}_i-
		\nabla_i{\rm tr}_g\dot{\hat g}\big)\,\mathrm d\sigma_g  +\int_{\partial M\cap\overline{\partial U\cap\mathring M}}
		\vartheta^\alpha\dot{\hat g}_{\alpha i}\mathbf n^i\,
		\mathrm d\theta_g.
	\end{align}
	Here $\mathbf n'$ is the outward unit normal to $M\cap U$ along $\partial U\cap \mathring{M}$, and $\vartheta$ is the outward conormal to $\partial M\cap \overline{\partial U\cap\mathring M}$ in $\partial M\cap U$.
	
	Since $K\Subset U$, we have $\dot{\hat g}=\frac{4}{n-2}\dot u g$ on $\partial U\cap \mathring{M}$. Therefore
	\begin{align*}
	\mathbf n'{}^i\left(\nabla_j\dot{\hat g}^j{}_i
	-\nabla_i{\rm tr}_g\dot{\hat g}\right)
	&=-\frac{4(n-1)}{n-2}\frac{\partial\dot u}{\partial\mathbf n'}
	=-a_n\frac{\partial\dot u}{\partial\mathbf n'},\\
	\dot{\hat g}_{\alpha i}\mathbf n^i
	&=\frac{4}{n-2}\dot u\,g_{\alpha i}\mathbf n^i=0.
	\end{align*}
	For every $t$, Proposition~\ref{Prop: the conformal factor}, applied with
	the metric $g_t$, gives
	\[
	\int_{\partial U\cap\mathring M}
	\frac{\partial u_t}{\partial\mathbf n'_{g_t}}\,
	\mathrm d\sigma_{g_t}=0.
	\]
	Since $u_0\equiv1$, taking the difference quotient along $t_j$ yields
	\[
	\int_{\partial U\cap\mathring M}
	\frac{\partial\dot u}{\partial\mathbf n'_g}\,
	\mathrm d\sigma_g=0.
	\]
	Consequently, 
	\begin{equation}\label{Eq: artificial boundary vanishes}
	\begin{aligned}
	&\int_{\partial U\cap\mathring M}\mathbf n'{}^i
	\big(\nabla_j\dot{\hat g}^j{}_i-
	\nabla_i{\rm tr}_g\dot{\hat g}\big)\,\mathrm d\sigma_g -\int_{\partial M\cap\overline{\partial U\cap\mathring M}}
	\vartheta^\alpha\dot{\hat g}_{\alpha i}\mathbf n^i\,
	\mathrm d\theta_g=0.
	\end{aligned}
	\end{equation}
	
	By \eqref{conf factor ut} and the conformal transformation laws,
	$R_{\hat g_t}=u_t^{-\frac{4}{n-2}}R_g$ and
	$H_{\hat g_t}=u_t^{-\frac{2}{n-2}}H_g$. Thus
	\begin{align*}
	\delta R_{\hat g}
	&=-\frac{4}{n-2}\dot u\,R_g,\\
	\delta(\mathrm d\mu_{\hat g})
	&=\frac12{\rm tr}_g\dot{\hat g}\,\mathrm d\mu_g
	=\left(\frac12{\rm tr}_gk+
	\frac{2n}{n-2}\dot u\right)\mathrm d\mu_g.
	\end{align*}
	Therefore,
	\begin{align*}
	\delta\left(R_{\hat g}\,\mathrm d\mu_{\hat g}\right)
	&=\left[-\frac{4}{n-2}\dot u\,R_g
	+\frac12R_g\left({\rm tr}_gk+
	\frac{4n}{n-2}\dot u\right)\right]\mathrm d\mu_g\\
	&=\left(2\dot uR_g+
	\frac12R_g\,{\rm tr}_gk\right)\mathrm d\mu_g.
	\end{align*}
	Moreover, using $\dot{\hat g}=k+\frac{4}{n-2}\dot u\,g$, we obtain
	\[
	\begin{aligned}
		\Big\langle{\rm Ric}_g-\frac12R_gg,\dot{\hat g}\Big\rangle
		&=\langle{\rm Ric}_g,k\rangle-\frac12R_g\,{\rm tr}_gk-2\dot uR_g.
	\end{aligned}
	\]  
	Hence
	\begin{align*}
	&2\dot uR_g+\frac12R_g\,{\rm tr}_gk
	+\Big\langle{\rm Ric}_g-\frac12R_gg,
	\dot{\hat g}\Big\rangle =\langle k,{\rm Ric}_g\rangle.
	\end{align*}
	
	Similarly,
	\begin{align*}
	\delta H_{\hat g}
	&=-\frac{2}{n-2}\dot u\,H_g,\\
	\delta(\mathrm d\sigma_{\hat g})
	&=\frac12{\rm tr}_{h_g}\dot{\hat g}^T\,\mathrm d\sigma_g =\left(\frac12{\rm tr}_{h_g}k^T
	+\frac{2(n-1)}{n-2}\dot u\right)\mathrm d\sigma_g,
	\end{align*}
	and hence
	\[
	\delta\left(2H_{\hat g}\,\mathrm d\sigma_{\hat g}\right)
	=\left(4\dot uH_g+H_g\,{\rm tr}_{h_g}k^T\right)\mathrm d\sigma_g.
	\]
	On the other hand,  
	\begin{align*}
	\left\langle A_g-H_gh_g,\dot{\hat g}^T\right\rangle
	&=\langle A_g,k^T\rangle-H_g\,{\rm tr}_{h_g}k^T
	-4\dot uH_g.
	\end{align*}
	Consequently,
	\begin{align*}
	&4\dot uH_g+H_g\,{\rm tr}_{h_g}k^T
	+\left\langle A_g-H_gh_g,\dot{\hat g}^T\right\rangle 
 =\langle k,A_g\rangle.
	\end{align*}
	
	Combining these identities with \eqref{Eq: artificial boundary vanishes},
	and applying the localized variation calculation along the sequence $t_j$,
	we obtain
	\begin{equation}\label{Eq: first variation rigidity}
	\begin{aligned}
	&2(n-1)\omega_{n-1}\lim_{j\to\infty}
	\frac{\mathfrak m(M,\hat g_{t_j},\mathcal E)
	-\mathfrak m(M,g,\mathcal E)}{t_j}\\
	&\qquad=\int_M\langle k,{\rm Ric}_g\rangle\,\mathrm d\mu_g
	+\int_{\partial M}\langle k^T,A_g\rangle\,\mathrm d\sigma_g.
	\end{aligned}
	\end{equation}
	The right-hand side is independent of the sequence and of the
	subsequential limit $\dot u$. Since
	\[
	\mathfrak m(M,\hat g_t,\mathcal E)\ge0,
	\qquad
	\mathfrak m(M,g,\mathcal E)=0,
	\]
	letting $t_j\to0^+$ and $t_j\to0^-$ in
	\eqref{Eq: first variation rigidity} gives
	\[
	0=\int_M\langle k,{\rm Ric}_g\rangle\,\mathrm d\mu_g
	+\int_{\partial M}\langle k^T,A_g\rangle\,\mathrm d\sigma_g.
	\]
	As $k$ is arbitrary, ${\rm Ric}_g=0$ and $A_g=0$.
	
	\medskip
	\textit{Step 2.}  $(M,g,\mathcal E)$ is isometric to the Euclidean half-space.
	
	It remains to identify the global geometry.
	Let $(\tilde M,\tilde g)$ denote the double of $(M,g)$ across $\partial M$. Then $(\tilde M,\tilde g)$ is complete with arbitrary ends, and the double $\tilde{\mathcal E}$ of $\mathcal E$ is a distinguished asymptotically flat end. Its mass satisfies
	\[
	\mathfrak m(\tilde M,\tilde g,\tilde{\mathcal E})=2\mathfrak m(M,g,\mathcal E)=0.
	\] 
	Since ${\rm Ric}_g=0$ and $A_g=0$, the doubled metric $\tilde g$ is $C^{2,\alpha}$ and Ricci-flat across the doubling hypersurface. Elliptic regularity for the Ricci-flat equation in harmonic coordinates then shows that $\tilde g$ is smooth. The rigidity theorem \cite[Theorem 1.2]{Zhu23} therefore yields
	\[
	(\tilde M,\tilde g)\cong(\mathbb R^n,\delta).
	\]
	Finally, the doubling involution corresponds to a reflection across an affine hyperplane in $\mathbb R^n$. Hence
	\[
	(M,g,\mathcal E)\cong(\mathbb R^n_+,\delta).
	\] 
\end{proof}

\end{document}